\documentclass[11pt,a4paper,reqno]{amsart}
\usepackage[T1]{fontenc}
\usepackage[utf8]{inputenc}
\usepackage{lmodern}
\usepackage{amsmath,amssymb,mathtools,mathrsfs}
\usepackage{microtype}
\usepackage[a4paper,margin=28mm,headheight=14pt,footskip=20pt]{geometry}
\usepackage{booktabs}
\usepackage{enumitem}
\usepackage{listings}
\usepackage{xcolor}
\usepackage[hidelinks,unicode]{hyperref}
\hypersetup{pdftitle={Spectral heat trace estimates and area gaps for minimal surfaces in spheres},pdfauthor={Peng Wang and Zhenxiao Xie},pdfsubject={Area gaps for minimal surfaces
in spheres}}
\numberwithin{equation}{section}
\newtheorem{theorem}{Theorem}[section]

\newtheorem{lemma}[theorem]{Lemma}
\newtheorem{corollary}[theorem]{Corollary}
\theoremstyle{definition}

\theoremstyle{remark}
\newtheorem{remark}[theorem]{Remark}
\DeclareMathOperator{\Area}{Area}
\DeclareMathOperator{\Vol}{Vol}

\newcommand{\ipL}[2]{\langle #1,#2\rangle_{L^2}}
\newcommand{\normL}[1]{\|#1\|_{L^2}}

\newcommand{\cmt}[1]{[\Delta,#1]}

\newcommand{\zz}{\mu}

\newcommand{\sph}{\mathbb S}
\newcommand{\T}{\mathbb T}

\newcommand{\dd}{\,d}
\newcommand{\blue}{\textcolor{blue}}

\allowdisplaybreaks[2]
\title{Area gaps for minimal surfaces 
in spheres}
\author{Peng Wang}
\address{School of Mathematics and Statistics, Key Laboratory of Analytical Mathematics and Applications (Ministry of Education), FJKLAMA, Fujian Normal University, 350117 Fuzhou, P.R. China}
\email{pengwang@fjnu.edu.cn, netwangpeng@163.com}

\author{Zhenxiao Xie}
\address{School of Mathematical Sciences, Beihang University,  Beijing 100191, P. R. China}
\email{xiezhenxiao@buaa.edu.cn}

\begin{document}

\begin{abstract}
In this paper, we prove that a linearly full minimal $2$-torus in $\mathbb{S}^N$ with $N\geq 5$ must have area $>2\pi^2$, and that a linearly full closed minimal surface in $\mathbb{S}^N$ with $N\geq 9$ must have area $>8\pi$. These results verify the Willmore conjecture and the generalized Willmore conjecture 
in higher codimension for minimal surfaces. The area gaps are obtained by applying the Cheng--Li--Yau comparison theorem and deriving a lower bound of the heat trace, where the monotonicity of the shifted normalized Riesz mean is used. We also establish a volume gap above the balanced Clifford product for linearly full closed minimal $n$-submanifolds in $\mathbb{S}^N$ when $n=3$, $N\geq 9$, or $n\geq 4$, $N\geq 3n-1$.
\end{abstract}
\maketitle


\section{Introduction}

The study of minimal surfaces in the unit round $N$-sphere $\sph^N$ is a classical and fundamental topic in geometry \cite{CM}. It is closely related to several other fields, 
including spectral geometry, elliptic partial differential equations, and conformal geometry. Among the many challenging problems in this field, obtaining sharp or quantitative area estimates for minimal surfaces has attracted considerable attention, a problem that is also closely tied to the estimates of the Willmore functional in conformal geometry. For closed minimal surfaces in $\mathbb{S}^3$ other than the round $2$-sphere, the area is at least $2\pi^2$, and equality holds if and only if the surface is congruent to the Clifford torus $\T^2_{CL}$. This follows as a corollary of several important works: Li--Yau's multiplicity estimate \cite{LiYau}, Brendle's solution of the Lawson conjecture \cite{Brendle} when the surface is a torus, and  Marques--Neves's proof of the Willmore conjecture~\cite{MN}. On the other hand, for every genus $\mathrm{g}\geq2$, the Lawson embedded minimal surface $\xi_{\mathrm{g},1}$ has area strictly less than $8\pi$ \cite{K-1989,K-1996,Lawson-1970}.

A fundamental tool for estimating the lower bound of the area (resp. volume) of minimal surfaces (resp. submanifolds) in $\sph^N$ is the heat trace comparison developed by Cheng--Li--Yau~\cite{CLY}, see the recent work of Ding--Ge--Li \cite{DGL} in this direction. The comparison reduces the geometric estimate to suitable spectral inequalities. Over the past decades, many important spectral inequalities have been developed through the efforts of numerous authors. 
For minimal submanifolds of spheres, relevant eigenvalue estimates were initiated by Yang–Yau ~\cite{YY} and subsequently refined by Leung~\cite{Leung} and Cheng--Yang~\cite{CY}. 

To  investigate these bounds systematically, an important approach based on commutator estimates was developed by Harrell--Michel \cite{HM} and Harrell--Stubbe~\cite{HS1997}, and later extended to elliptic and submanifold settings by Ashbaugh--Hermi \cite{AH} and El Soufi--Harrell--Ilias~\cite{ESH}. 
Differential inequalities and monotonicity formulas for Riesz means were developed further in the works of Harrell--Hermi \cite{HH2008,HH} and Harrell--Stubbe~\cite{HS2011}. 

The Riesz mean is intimately connected to the heat trace via the Laplace transform. In this paper, by using the monotonicity of the shifted normalized second-order Riesz mean and bounding it from below by an explicit quadratic--cubic tangent
polynomial, we obtain a lower bound for the heat trace. Together with the Cheng--Li--Yau heat trace comparison \cite{CLY}, this yields the following lower bound on the volume of minimal submanifolds. 

\begin{theorem}\label{thm:spectral}
Let $f:M^n\to\mathbb{S}^N(1)$ be a closed minimal immersion with $n\geq2$. Suppose that its $m$-th eigenvalue $\lambda_m$, counted with multiplicity, satisfies $\lambda_m\leq n$. Then 
\begin{equation}\label{eq:general-area}
 \begin{split}
\Vol(f)&\geq
 \omega_n\, \sup_{t>0}
 \frac{1+me^{-nt}+(b_{n,m}+3c_{n,m}/t)e^{-s_{n,m}t}}{Z_{\sph^n}(t)},
 \end{split}
\end{equation}
where  $Z_{\sph^n}(t)$ is the heat trace of $\sph^n$, $\omega_n\triangleq \Vol(\sph^n)$, $a_n\triangleq n^2/4$, and 
\begin{equation}\label{eq:intro-parameters}
 \begin{split}
s_{n,m}\triangleq \frac{2m(n+2)}{m+1}, \qquad
 b_{n,m}&\triangleq \frac{n((n+4)m-n)}{4(s_{n,m}+a_n)},\qquad
 c_{n,m}\triangleq \frac{mn(n+2)(n+4)}{12(s_{n,m}+a_n)^2}.
 \end{split}
\end{equation}
\end{theorem}
Our main focus is the case $n=2$, motivated  
by our interest in the Willmore conjecture in higher codimension. By 
choosing an appropriate $t$ in \eqref{eq:general-area}, we obtain the following area gaps.
\begin{theorem}\label{thm:area}
Let $f:M^2\to\sph^N$ be a closed, linearly full minimal immersion. 
\begin{enumerate} 
    \item[{\rm (1)}] If $N\geq9$, then
\begin{equation}\label{eq:all-surface-area}
 \Area(f)
 >8\pi.
\end{equation}
    \item[{\rm (2)}] If $M^2$ is a torus and $N\geq 5$, then 
\begin{equation}\label{eq:torus-area}
 \Area(f)>
 2\pi^2=\Area(\T^2_{CL}).
\end{equation}
\item[{\rm (3)}] 
 If $M^2$ is a Klein bottle and $N\geq 6$, then
\begin{equation}\label{eq:klein-area}\Area(f)>
6\pi E\left(\frac{2\sqrt2}{3}\right)=\Area(\tilde{\tau}_{3,1}),\end{equation}
where $\tilde{\tau}_{3,1}$ is the Lawson bipolar Klein bottle in $\sph^4$, and
$E(k)=\int_0^{\pi/2}\sqrt{1-k^2\sin^2\theta}\dd\theta$ is the complete
elliptic integral of the second kind. \end{enumerate}    
\end{theorem}

We point out that the gap \eqref{eq:torus-area} verifies the Willmore conjecture for linearly full minimal tori of codimension at least $3$, the gap \eqref{eq:all-surface-area} verifies the generalized Willmore conjecture
for linearly full minimal surfaces of codimension at least $7$, and the gap \eqref{eq:klein-area} verifies the Hirsch--M\"{a}der--Baumdicker conjecture for linearly full minimal Klein bottles of codimension at least $4$.

Combined with the Willmore stability argument of Kusner and Wang \cite{KW} regarding minimal $2$-tori in $\sph^N$, Theorem \ref{thm:area} implies that if the Willmore minimizer for $2$-tori in $\sph^N$ ($N\geq5$) is achieved by a minimal $2$-torus, it must be the Clifford torus $\mathbb{T}^2_{CL}$. 
This observation 
suggests that one potential strategy is to determine whether Willmore minimizers can be realized by minimal surfaces in spheres.  


For linearly full closed minimal submanifolds of dimension $n\geq3$ in $\mathbb{S}^N$, we also prove that the volume is strictly larger than that of the balanced Clifford product, provided $N\geq9$ for $n=3$, and $N\geq3n-1$ for $n\geq4$; see Theorem~\ref{thm:clifford-volume-gap}.  This is motivated by the famous Solomon-Yau conjecture concerning the second volume gap of all minimal hypersurfaces in $\sph^{n+1}$ \cite{Yau}. See Ilmanen-White \cite{IW}  for some partial results by using mean curvature flow methods. We also refer to Ge and Li's works  \cite{GL0, GL}  concerning the volume gap for linearly full minimal submanifolds via curvature weights, where they use different methods involving spherical height functions, Simons' equality and a mean-value formula on the Euclidean minimal cone.

The same heat-trace estimate also shows that, for fixed intrinsic
dimension, the volume of a linearly full minimal submanifold grows
at least linearly with its spherical codimension.
\vspace{2mm}

{\bf {Plan of the paper.}} Section~\ref{sec:commutator} collects the spectral preliminaries for minimal submanifolds in $\mathbb{S}^N$, 
including Cheng--Li--Yau's heat-trace comparison and Riesz-mean monotonicity. In Section~\ref{sec:propagation}, we derive the 
heat-trace estimate and then specialize to the volume theorems (see Theorem~\ref{thm:spectral} and Theorem~\ref{thm:clifford-heat-bound}) and their applications (see Theorem~\ref{thm:area} and Theorem~\ref{thm:clifford-volume-gap}). Section~\ref{sec:clifford-volume-gap} gives the higher-dimensional Clifford-volume gap. Appendix~\ref{sec:certificates} records the surface constant estimates.

\section{Spectral preliminaries for minimal submanifolds}
\label{sec:commutator}

 We will collect some spectral properties for a minimal submanifold $M^n$ in $\sph^N$, including Cheng--Li--Yau's heat-trace comparison and Riesz-mean monotonicity. 
\subsection{Cheng--Li--Yau's heat-trace comparison}
\label{subsec:submanifold-preliminaries}~\vspace{1mm}


Throughout the paper, $M^n$ is  assumed to be  a connected, $n$-dimensional closed manifold unless explicitly stated otherwise. Let
$f:M^n\to\sph^N$ be a minimal immersion with induced metric $g=f^*g_{\sph^N}$ and the corresponding volume density $dV_g$.  Denote by 
$
 \Vol(f)=\int_{M^n} dV_g
$
the volume of $f$, which will be referred to as area and denoted by $\Area(f)$ when $n=2$.  
For real functions $u,v$, we use
\[
 \ipL{u}{v}\triangleq \int_{M^n} uv\dd V_g,\qquad
 \normL{u}^2\triangleq \ipL{u}{u},
\]
whereas pointwise inner products of tangent vectors are written as $g(X,Y)$.
 
On $M^n$, define the Laplace--Beltrami operator of $g$ and its eigenvalues as follows:
\begin{equation}\label{eq:indexing}
 \Delta=\operatorname{div}_g\nabla,\qquad
 \Delta u_j+\lambda_j u_j=0,\qquad
 0=\lambda_0<\lambda_1\leq\lambda_2\leq\cdots,
\end{equation}
where eigenvalues are counted according to their 
multiplicities. 
Write
\[
 E(\lambda)=\{u\in C^\infty(M^n,\mathbb R):\Delta u+\lambda u=0\},
\]
and denote the heat trace by
\[
 Z_M(t)=\sum_{j=0}^{\infty}e^{-t\lambda_j}.
\]
For the unit 
$n$-sphere 
$\sph^n$, set 
\[
 \omega_n\triangleq \Vol(\sph^n),\qquad Z_n(t)\triangleq Z_{\sph^n}(t).
\]
In particular, when $n=2$, we have $\omega_2=4\pi$, and the heat trace of  $\sph^2$ is 
\begin{equation}\label{eq:sphere-trace}
 Z_2(t)=\sum_{\ell=0}^{\infty}(2\ell+1)e^{-t\ell(\ell+1)} .
\end{equation}

By Takahashi's theorem~\cite{Takahashi}, the coordinate functions of $f$ satisfy
\begin{equation}\label{eq:mini}
 \Delta x_\alpha+nx_\alpha=0,\qquad 1\leq\alpha\leq N+1.
\end{equation}
The isometric immersion identities read
\begin{equation}\label{eq:coordinate-identities}
 \sum_\alpha x_\alpha^2=1,\qquad
 \sum_\alpha x_\alpha\nabla x_\alpha=0,\qquad
 \sum_\alpha dx_\alpha\otimes dx_\alpha=g.
\end{equation}
Here and below $\sum_\alpha$ means $\sum_{\alpha=1}^{N+1}$. The last identity gives
\[
 |X|_g^2=\sum_\alpha g(\nabla x_\alpha,X)^2
 \qquad\text{for every }X\in\mathfrak X(M).
\]
If $f$ is linearly full, the coordinate functions are linearly independent. Thus
\begin{equation}\label{eq:coordinate-spectrum}
 \dim E(n)\geq N+1,\qquad \lambda_{N+1}\leq n.
\end{equation}
This indicates
\begin{equation}\label{eq:coordinate-heat-lower}
 Z_M(t)>1+(N+1)e^{-nt},\qquad t>0.
\end{equation}

In \cite{CLY}, Cheng--Li--Yau obtained the following heat trace comparison theorem. 
\begin{theorem}\label{thm:cly}
For a closed minimal immersion $f:M^n\to\sph^N$ and $t>0$,
\begin{equation}\label{eq:cly}
 Z_M(t)\leq\frac{\Vol(f)}{\omega_n}Z_n(t).
\end{equation}
\end{theorem}
Theorem \ref{thm:cly} is also used in \cite{Haskins} to estimate area of minimal Legendrian 2-tori.

\subsection{Normalized Riesz-mean monotonicity}
\label{subsec:submanifold-yang}~

For a smooth function $\varphi$, consider the commutator
\[
 \cmt{\varphi}u\triangleq \Delta(\varphi u)-\varphi\Delta u
 =u\Delta\varphi+2g(\nabla\varphi,\nabla u),\qquad u\in C^2(M).
\]
introduced in \cite{AH, ESH, HS1997}. 
It has been used to build an eigenvalue-cutoff formulation; see Ashbaugh--Hermi~\cite[Theorem~2.1]{AH}, El Soufi--Harrell--Ilias~\cite[Lemma~2.1]{ESH}, Harrell--Stubbe~\cite[Theorem~5]{HS1997} and  \cite[Theorem~2.1]{HS2011}. Here we include a proof for completeness.
\begin{lemma}\label{lem:cutoff}
For any real smooth function $\varphi\in C^\infty(M)$ and any $\zz\geq0$,
\begin{equation}\label{eq:commutator-cutoff}
 \begin{aligned}
 &-\sum_{\lambda_i<\zz}(\zz-\lambda_i)^2\ipL{\cmt{\varphi}u_i}{\varphi u_i}\leq\sum_{\lambda_i<\zz}(\zz-\lambda_i)\normL{\cmt{\varphi}u_i}^2,
 \end{aligned}
\end{equation}
where $\{u_i\}_{i\geq0}$ denotes an orthonormal basis of $E(\lambda_i)$.
\end{lemma}
\begin{proof}
For a fixed $\varphi$, assume $\varphi u_i=\sum_{j}a_{ij}u_j$ with
$
 a_{ij}\triangleq \ipL{\varphi u_i}{u_j}
 =\int_M\varphi u_i u_j\dd V_g$ and $a_{ij}=a_{ji}$. Then we have
 $\cmt{\varphi}u_i=\sum_j(\lambda_i-\lambda_j)a_{ij}u_j.$
Parseval's identity therefore yields 
\[
 \sum_j(1+\lambda_j^2)a_{ij}^2
 =\normL{\varphi u_i}^2+\normL{\Delta(\varphi u_i)}^2<\infty,\]
 \[\ipL{\cmt{\varphi}u_i}{\varphi u_i}
 =\sum_j(\lambda_i-\lambda_j)a_{ij}^2,\qquad \normL{\cmt{\varphi}u_i}^2=\sum_j(\lambda_i-\lambda_j)^2a_{ij}^2.\]
So the right-hand side of \eqref{eq:commutator-cutoff} minus its left-hand side is
\begin{equation}\label{eq:pairwise-cancellation}
 \sum_{\lambda_i<\zz}\sum_j(\zz-\lambda_i)(\lambda_j-\lambda_i)(\lambda_j-\zz)a_{ij}^2.
\end{equation}
If both $\lambda_i<\zz$ and $\lambda_j<\zz$, either the terms with indices $(i,j)$ and $(j,i)$ cancel or the diagonal terms vanish. So the remaining terms have $\lambda_i<\zz\leq\lambda_j$, and all three factors in \eqref{eq:pairwise-cancellation} are nonnegative. This proves the inequality.
\end{proof}

For $\zz\geq0$, the first-order and second-order Riesz-means are defined as 
\[
 R_1(\zz)=\sum_i(\zz-\lambda_i)_+,\qquad
 R_2(\zz)=\sum_i(\zz-\lambda_i)_+^2.
\]
Now we state the Yang-type inequality and the shifted normalized Riesz-mean monotonicity, together with a proof for completeness. These results have been built in   Cheng--Yang~\cite{CY}, El Soufi--Harrell--Ilias~\cite[Theorem~3.1(I)]{ESH}, Harrell--Stubbe~\cite[Theorem~4.1]{HS2011}. 
\begin{theorem}\label{thm:yang}
Let $f:M^n\to\sph^N$ be a closed minimal immersion with $n\geq2$.
\begin{enumerate}
    \item[{\rm(1)}] For every $\zz\geq0$,
\begin{equation}\label{eq:yang-general}
 R_2(\zz)\leq
 \frac4n\sum_{\lambda_i<\zz}(\zz-\lambda_i)
 \left(\lambda_i+\frac{n^2}{4}\right).
\end{equation}
\item[{\rm(2)}] The function
\begin{equation}\label{eq:monotone-quotient}
 \zz\longmapsto\frac{R_2(\zz)}{(\zz+n^2/4)^{2+n/2}}
\end{equation}
is nondecreasing on $[0,\infty)$.
\end{enumerate} 
\end{theorem}
\begin{proof}
(1) Substituting $x_{\alpha}$ into \eqref{eq:mini} yields
\[\cmt{x_\alpha}u_i=-nx_\alpha u_i+2g(\nabla x_\alpha,\nabla u_i).\]
Multiplication by $x_\alpha u_i$ and summation, together with \eqref{eq:coordinate-identities}, give
\begin{align*}
 \sum_\alpha(\cmt{x_\alpha}u_i)x_\alpha u_i
 &=-nu_i^2\sum_\alpha x_\alpha^2
 +2u_i g\left(\sum_\alpha x_\alpha\nabla x_\alpha,\nabla u_i\right)=-nu_i^2.
\end{align*}
Integration yields
\begin{equation}\label{eq:first-commutator}
 \sum_\alpha\ipL{\cmt{x_\alpha}u_i}{x_\alpha u_i}=-n.
\end{equation}
Similarly, \eqref{eq:coordinate-identities} again gives
\begin{align*}
 \sum_\alpha(\cmt{x_\alpha}u_i)^2
 &=n^2u_i^2\sum_\alpha x_\alpha^2
 -4nu_i g\left(\sum_\alpha x_\alpha\nabla x_\alpha,\nabla u_i\right)
 +4\sum_\alpha g(\nabla x_\alpha,\nabla u_i)^2\\
 &=n^2u_i^2+4|\nabla u_i|_g^2.
\end{align*}
Since $\int_M|\nabla u_i|_g^2\dd V_g=-\ipL{\Delta u_i}{u_i}=\lambda_i$, we obtain
\begin{equation}\label{eq:second-commutator}
 \sum_\alpha\normL{\cmt{x_\alpha}u_i}^2=n^2+4\lambda_i.
\end{equation}
Set $\varphi=x_\alpha$ in \eqref{eq:commutator-cutoff} and sum over $\alpha$. By \eqref{eq:first-commutator} and \eqref{eq:second-commutator},
\[n\sum_{\lambda_i<\zz}(\zz-\lambda_i)^2\leq\sum_{\lambda_i<\zz}(\zz-\lambda_i)(n^2+4\lambda_i),\] which is exactly \eqref{eq:yang-general}.

(2) By \eqref{eq:yang-general}, \[
 R_2(\zz)\leq\frac4n\left((\zz+n^2/4)R_1(\zz)-R_2(\zz)\right).
\]
Since each function $(\zz-\lambda_i)_+^2$ is continuously differentiable and
the sum is locally finite, $R_2'=2R_1$, including at the eigenvalues. Hence
$(2+n/2)R_2(\zz)\leq(\zz+n^2/4)R_2'(\zz)$, and therefore
\[
 \frac{d}{d\zz}\frac{R_2(\zz)}{(\zz+n^2/4)^{2+n/2}}
 =\frac{(\zz+n^2/4)R_2'(\zz)-(2+n/2)R_2(\zz)}
 {(\zz+n^2/4)^{3+n/2}}\geq0.
\]
Continuity at $\zz=0$ completes the proof.
\end{proof}

\section{Heat trace estimates and geometric bounds}\label{sec:propagation}

We first turn finite low-spectrum information into an explicit heat-trace
estimate for all minimal submanifolds, which also yields Theorem \ref{thm:spectral}. Then we specialize to minimal surfaces for
the area bounds and their applications. Finally we apply the heat-trace
estimate to minimal submanifolds of dimension $n\geq3$. 

\subsection{Heat-trace estimate and volume theorem for minimal submanifolds}

\begin{theorem}\label{thm:general-heat}
Let $f:M^n\to\sph^N$ be a closed minimal immersion with $n\geq2$ and
$\lambda_m\leq n$ for a positive integer $m$. Then, with the parameters $b_{n,m}$, $c_{n,m}$ and $s_{n,m}$ in \eqref{eq:intro-parameters},  for every $t>0$, 
\begin{equation}\label{eq:general-heat}
 Z_M(t)>G_{n,m}(t)\triangleq 
 1+me^{-nt}+\left(b_{n,m}+\frac{3c_{n,m}}{t}\right)e^{-s_{n,m}t}.
\end{equation}
\end{theorem}
\begin{proof}
For simplicity, set $a=a_n=n^2/4$, $p=2+n/2$, and $s=s_{n,m}>n$.
Since $\lambda_0=0$ and at least $m$ positive eigenvalues are at most $n$,
\begin{equation}\label{eq:m-seed}
 R_2(\zz)\geq P_{n,m}(\zz)\triangleq \zz^2+m(\zz-n)_+^2.
\end{equation}
For $\zz\geq n$, write
$P_{n,m}(\zz)=P(\zz)\triangleq (m+1)\zz^2-2mn\zz+mn^2$.
A direct computation gives
\begin{equation}\label{eq:tangency-general}
 pP(\zz)-(\zz+a)P'(\zz)
 =\frac n2(m+1)(\zz-n)(\zz-s),\qquad P(s)=4m(s+a).
\end{equation}
Set
\[
 \kappa_{n,m}\triangleq \frac{P(s)}{(s+a)^p}=\frac{4m}{(s+a)^{p-1}}.
\]
By   Theorem~\ref{thm:yang}, for every $\zz\geq s$,
\begin{equation}\label{eq:tail-cubic-lower}
 \frac{R_2(\zz)}{(\zz+a)^p}
 \geq\frac{R_2(s)}{(s+a)^p}
 \geq\frac{P(s)}{(s+a)^p}=\kappa_{n,m}.
\end{equation}
Thus
\begin{equation}\label{eq:R2}
 R_2(\zz)\geq
 \begin{cases}
 P_{n,m}(\zz),&0\leq\zz\leq s,\\
 \kappa_{n,m}(\zz+a)^p,&\zz\geq s.
 \end{cases}
\end{equation}

For $\zz\geq s$, set $Q(\zz)=\kappa_{n,m}(\zz+a)^p-P(\zz)$.
Equation~\eqref{eq:tangency-general} gives $Q(s)=Q'(s)=0$, and
\begin{equation*} 
 \begin{split}
 \frac12Q''(s)&=\frac{2mp(p-1)}{s+a}-(m+1)=b_{n,m},\\
 \frac16Q'''(s)&=\frac{2mp(p-1)(p-2)}{3(s+a)^2}=c_{n,m}.
 \end{split} 
\end{equation*}
Since $p\geq3$, we have $Q'''(\zz)\geq Q'''(s)$ for $\zz\geq s$.
Taylor's formula with integral remainder therefore gives
\begin{equation}\label{eq:general-factorization}
 \begin{split}
 Q(\zz)&=b_{n,m}(\zz-s)^2+
 \frac12\int_s^{\zz}(\zz-v)^2Q'''(v)\dd v\geq b_{n,m}(\zz-s)^2+c_{n,m}(\zz-s)^3.
 \end{split}
\end{equation}
Together with \eqref{eq:R2}, we get 
\begin{equation}\label{eq:global-cubic-lower}
 R_2(\zz)\geq\zz^2+m(\zz-n)_+^2
 +b_{n,m}(\zz-s)_+^2+c_{n,m}(\zz-s)_+^3.
\end{equation}

For $t>0$, $\lambda\geq0$, and $k=2,3$, the substitution
$y=\zz-\lambda$ gives
\[
 \int_0^\infty e^{-t\zz}(\zz-\lambda)_+^k\,\dd\zz
 =e^{-t\lambda}\int_0^\infty e^{-ty}y^k\,\dd y
 =\frac{k!}{t^{k+1}}e^{-t\lambda}.
\]
Since all summands in $e^{-t\zz}R_2(\zz)$ are nonnegative,
Tonelli's theorem yields
\begin{equation}\label{eq:heat-riesz}
 \begin{aligned}
  \int_0^\infty e^{-t\zz}R_2(\zz)\,\dd\zz
  &=\sum_{j=0}^\infty
    \int_{\lambda_j}^\infty
    e^{-t\zz}(\zz-\lambda_j)^2\,\dd\zz
=\frac{2}{t^3}\sum_{j=0}^\infty e^{-t\lambda_j}
   =\frac{2}{t^3}Z_M(t).
 \end{aligned}
\end{equation}
Multiplying \eqref{eq:global-cubic-lower} by
$\frac{t^3}{2}e^{-t\zz}$ and integrating over $[0,\infty)$,
we therefore obtain
\[
 Z_M(t)\geq1+me^{-nt}
 +\left(b_{n,m}+\frac{3c_{n,m}}t\right)e^{-st}=G_{n,m}(t).
\]

It remains to prove strictness. If $n>2$, then $p>3$ and
$Q'''(v)>Q'''(s)$ for $v>s$, so \eqref{eq:general-factorization} is strict
for $\zz>s$. If $n=2$, then $p=3$ and
\eqref{eq:general-factorization} is an identity. In this case choose a
nonempty open interval $I\subset(s,\infty)$ containing no eigenvalue.
On $I$, $R_2$ is a quadratic polynomial, whereas
$\kappa_{2,m}(\zz+1)^3$ is not quadratic. Their nonnegative difference
is therefore positive on a nonempty open subset of $I$. In either case,
\eqref{eq:heat-riesz} gives the strict inequality.
\end{proof}

\begin{remark}\label{rem:sm-history}
For $n=2$ write 
\[G_{2,m}=G_m,\quad P_{2,m}=P_m, \quad s_m=s_{2,m}, \quad b_m=b_{2,m}   \hbox{ and }  c_m=c_{2,m}.\] Then
$\kappa_{2,m}=c_m$, and the last two terms of \eqref{eq:global-cubic-lower}
are exactly the cubic correction
\[
 c_m(\zz+1)^3-P_m(\zz)
 =(\zz-s_m)^2\bigl(b_m+c_m(\zz-s_m)\bigr),\qquad \zz\geq s_m.
\]
The choice of $s_m$ is consistent with the finite-sum optimization in
Harrell--Stubbe~\cite[proof of Corollary~4.2,
equations~(4.6)--(4.9)]{HS2011}.
\end{remark}

\begin{proof}[Proof of Theorem~\ref{thm:spectral}]
By Theorems~\ref{thm:general-heat} and~\ref{thm:cly},
$
 G_{n,m}(t)<Z_M(t)\leq\frac{\Vol(f)}{\omega_n}Z_n(t).
$
Then 
\begin{equation}
    \label{eq:Vol}
\Vol(f)>\omega_n\frac{G_{n,m}(t)}{Z_n(t)}, 
\end{equation}and  taking the
supremum yields \eqref{eq:general-area}.
  \end{proof}

\subsection{Applications to minimal surfaces}\label{subsec:surface-applications}~
\vskip 0.1cm

We now specialize to $n=2$. Thus $dV_g=dA_g$ and $\Vol(f)=\Area(f)$. 
The  heat-trace comparison reads (recall \eqref{eq:sphere-trace})
\begin{equation}\label{eq:cly-surface}
 Z_M(t)\leq\frac{\Area(f)}{4\pi}Z_2(t).
\end{equation}
Theorem~\ref{thm:yang} specializes to
\begin{equation}\label{eq:yang}
 R_2(\zz)\leq2\sum_{\lambda_i<\zz}(\zz-\lambda_i)(\lambda_i+1),\qquad
 \frac{d}{d\zz}\frac{R_2(\zz)}{(\zz+1)^3}\geq0.
\end{equation}


\begin{proof}[Proof of Theorem~\ref{thm:area}] By \eqref{eq:coordinate-spectrum}, every linearly full closed minimal surface satisfies $\lambda_{N+1}\leq2$. More generally, Theorems~\ref{thm:general-heat} and~\ref{thm:cly} give 
\[\Area(f)>\mathcal B_{2,m}(t) \hbox{ for every }t>0\] 
whenever $\lambda_m\leq2$, where  
\begin{equation}\label{eq:dimension-bound-family}
\mathcal B_{n,m}(t)\triangleq \omega_n\frac{G_{n,m}(t)}{Z_n(t)}=\omega_n \frac{1+me^{-nt}+(b_{n,m}+3c_{n,m}/t)e^{-s_{n,m}t}}{Z_n(t)}.
 \end{equation}
We use the following fixed-time values:
\begin{equation}\label{eq:cert-fixed-times}
 C_7\triangleq \mathcal B_{2,7}(1/2),\qquad
 C_8\triangleq \mathcal B_{2,8}(1/2),\qquad
 C_{10}\triangleq \mathcal B_{2,10}(2/5).
\end{equation}

(1)    If $N\geq9$, then $\lambda_{10}\leq2$. So, together with  Lemma~\ref{lem:cseven}, we have $\Area(f)>C_{10}>8\pi$.

(2) For $N\geq6$, $\lambda_{7}\leq2$. So, together with  Lemma~\ref{lem:cseven}, we have $\Area(f)> C_7>2\pi^2$. Now consider $N=5$. If $\lambda_1<2$, then $\lambda_7\leq2$, and this fixed-time estimate and Lemma~\ref{lem:cseven} give  $\Area(f)>  C_7$. If $\lambda_1=2$, fullness in $\sph^5$ and \cite[Theorem~2.1]{ESI} imply that $f$ is congruent to the equilateral torus in $\sph^5$ with area $\Area(f)=\frac{4\pi^2}{\sqrt3}>20>C_7$.

(3) If $\lambda_1=2$, 
then $f$ is given by its first eigenfunctions and hence by \cite[Theorem~1.2]{ESGJ} 
$f$ is congruent to the Lawson bipolar Klein bottle in
$\mathbb S^4$ contradicting the linear fullness of $f$ in $\sph^N$ with $N\geq6$. Therefore,
$
\lambda_1<2$. Since the eigenvalue $2$ has multiplicity at least $N+1\geq7$,
we obtain $
\lambda_8\leq2$. So
$\Area(f)>C_8.$
By Lemma \ref{lem:cseven},
$C_8>\frac{48\pi}{7} >
6\pi E\left(\frac{2\sqrt2}{3}\right)$.
\end{proof}
 
Theorems~\ref{thm:general-heat} can also provide the following general area estimate.
\begin{corollary}
Let $f:M^2\to S^N$ be a closed linearly full minimal immersion. Then 
    \begin{equation}
        \label{eq:linear}
\operatorname{Area}(f)
>
\frac{9\pi}{20}(N+5).
  \end{equation}
\end{corollary}
\begin{proof}
First notice that by \eqref{eq:sphere-trace}, we have
$$Z_{2}(t)=\sum_{\ell=0}^{\infty}(2\ell+1)e^{-t\ell(\ell+1)}\leq \sum_{\ell=0}^{\infty} (2\ell+1)e^{-2t\ell} = \frac{1+e^{-2t}}{(1-e^{-2t})^2}.$$
Since $\lambda_{N+1}\leq 2$, Theorems~\ref{thm:general-heat} and~\ref{thm:cly} imply that,  for every $t\in(0,\infty)$,
\[
\frac{\operatorname{Area}(f)}{4\pi}
\ge \frac{Z_M(t)}{Z_2(t)}>\frac{1+(N+1)e^{-2t}}{Z_2(t)}\geq\left(1+(N+1)e^{-2t}\right)\frac{(1-e^{-2t})^2}{1+e^{-2t}}.
\]
Choosing $e^{-2t}=1/4$, we obtain \eqref{eq:linear}.

\end{proof}

For tori or Klein bottles, Nadirashvili ~\cite[Theorem~1(A)]{Nadirashvili} (see also  \cite{Berdnikov}) gives the following multiplicity bounds 
\begin{equation}\label{eq:nadirashvili}
 \dim E(\lambda_k)\leq2k+4\ \text{on }\T^2,
 \qquad
 \dim E(\lambda_k)\leq2k+3\ \text{on a Klein bottle},
\end{equation}
which  can improve Theorem~\ref{thm:area} a bit more.

\begin{corollary}\label{cor:refinements}
Let $f:M^2\to\sph^N$ be a linearly full minimal immersion. 
\begin{enumerate}
 \item[{\rm (1)}] If $M^2=\mathbb{T}^2$ and $N\ge8$, then $\Area(f)> C_{10}>8\pi$;
 \item[{\rm (2)}] If $M^2$ is a Klein bottle and $N\ge7$, then $\Area(f)> C_{10}>8\pi$;
 \item[{\rm (3)}] If $M^2$ is a Klein bottle and $N\ge5$, then $\Area(f)>C_7>2\pi^2$.
\end{enumerate}
\end{corollary}
\begin{proof}
Set $m\triangleq \min\{j\ge1:\lambda_j=2\}$. So $\dim E(\lambda_m)\geq N+1$, together with
\eqref{eq:nadirashvili}, we get
\begin{enumerate}
  \item[{\rm (1)}] $m\geq3$ when $N\geq 8$, and then $N+m\geq11$,  so $\lambda_{10}\leq\lambda_{11}\leq2$.
 \item[{\rm (2)}] $m\geq3$ when $N\geq 7$, and then $N+m\geq10$,  so $\lambda_{10}\leq2$.
 \item[{\rm (3)}] $m\geq2$ when $N\geq 5$, and then $N+m\geq7$, so $\lambda_7\leq2$.
 \end{enumerate}
The asserted area bounds follow from Theorems~\ref{thm:general-heat} and~\ref{thm:cly}, together with Lemma~\ref{lem:cseven}.
\end{proof}

\subsection{A Clifford-volume gap for higher-dimensional minimal submanifolds}~
\label{sec:clifford-volume-gap}
\vskip 0.1cm

For positive integers $p,q$ with $p+q=n$, the Clifford hypersurface $ \mathrm{Cl}_{p,q}$ is written as
\[
 \mathrm{Cl}_{p,q}\triangleq 
 \sph^p\!\left(\sqrt{\frac pn}\right)
 \times
 \sph^q\!\left(\sqrt{\frac qn}\right)
 \subset\sph^{n+1}(1),
~\hbox{ with }
 \Vol(\mathrm{Cl}_{p,q})=
 \omega_p\omega_q
 \left(\frac pn\right)^{p/2}
 \left(\frac qn\right)^{q/2}.
\]
Set
\begin{equation}\label{eq:clifford-balanced-volume}
 V_n^{\mathrm{Cl}}\triangleq \begin{cases}
 \Vol(\mathrm{Cl}_{k,k})=2^{-k}\omega_k^2,&n=2k,\\[1ex]
 \displaystyle
 \Vol(\mathrm{Cl}_{k,k+1})=\omega_k\omega_{k+1}
 \left(\frac{k}{2k+1}\right)^{k/2}
 \left(\frac{k+1}{2k+1}\right)^{(k+1)/2},&n=2k+1.
 \end{cases}
\end{equation}

{Before establishing the Clifford-volume gap, we first generalize the notion of linear area growth to the volume of general minimal submanifolds.}
\begin{theorem}\label{thm:clifford-heat-bound}
For a closed minimal immersion $f:M^n\to\sph^N$ with $n\geq3$ and
$\lambda_m\leq n$, we have, for every $t>0$,
\begin{equation}\label{eq:clifford-coordinate-heat}
 1+me^{-nt}<Z_M(t)\leq\frac{\Vol(f)}{\omega_n}Z_n(t).
\end{equation}
Consequently, for every $0<x<1$,
\begin{equation}\label{eq:clifford-basic-volume}
 \frac{\Vol(f)}{\omega_n}
 >\frac{(1+mx)(1-x)^n}{1+x}.
\end{equation}
In particular, for a linearly full immersion,  we have
\begin{equation}   \label{eq:CM}\frac{\operatorname{Vol}(f)}{\omega_n}
>\left(\frac{n}{n+1}\right)^n
\frac{N+n+2}{n+2}.
\end{equation}

\end{theorem}
\begin{proof}
The strict lower bound in \eqref{eq:clifford-coordinate-heat} follows from Theorem \ref{thm:general-heat} whereas  the upper bound follows from Theorem~\ref{thm:cly}. 

The eigenvalues of the unit round $n$-sphere $\sph^n$ are $\ell(\ell+n-1)$,
with multiplicities
\begin{equation}\label{eq:clifford-multiplicity}
 d_{\ell,n}=\frac{(2\ell+n-1)(\ell+n-2)!}{\ell!(n-1)!},
 \qquad \ell\geq0.
\end{equation}
Their generating function is
\begin{equation}\label{eq:clifford-generating-function}
 \sum_{\ell=0}^{\infty}d_{\ell,n}\zz^\ell
 =\frac{1+\zz}{(1-\zz)^n},\qquad |\zz|<1.
\end{equation}
Since $\ell(\ell+n-1)\geq n\ell$, putting $x=e^{-nt}$ gives
\begin{equation}\label{eq:clifford-sphere-upper}
 Z_n(t)=\sum_{\ell=0}^{\infty}d_{\ell,n}e^{-t\ell(\ell+n-1)}
 \leq\sum_{\ell=0}^{\infty}d_{\ell,n}x^\ell
 =\frac{1+x}{(1-x)^n}.
\end{equation}
Combining this estimate with \eqref{eq:clifford-coordinate-heat} proves
\eqref{eq:clifford-basic-volume}.  

In the linearly full case,  taking $m=N+1$ and  $x=(n+1)^{-1}$ in
\eqref{eq:clifford-basic-volume}, we obtain \eqref{eq:CM}.
\end{proof}
\begin{remark} Colding–Minicozzi's confinement theory \cite{CM2} gives codimension-independent local upper volume control and Euclidean volume growth for stationary varifolds. The estimates \eqref{eq:linear} and \eqref{eq:CM} provide global lower bounds for the densities of smooth linearly full minimal cones through their spherical links. Note that the coefficients in \eqref{eq:linear} and \eqref{eq:CM} are not sharp and can be further improved using Theorem \ref{thm:general-heat}, as shown in Theorems  \ref{thm:area} and \ref{thm:clifford-volume-gap}.
\end{remark}

The following technical lemma will be used to establish the Clifford-volume estimate.
\begin{lemma}\label{lem:clifford-volume-ratio}

Set $\rho_n\triangleq V_n^{\mathrm{Cl}}/\omega_n$. we have 
\[\rho_3=\frac8{3\sqrt3},
 \quad \rho_n\le\frac32 \hbox{ when }n\ge4,\quad \rho_n\le\frac{231}{160} \hbox{ when } n\ge12.\]
\end{lemma}
\begin{proof}
The explicit formulas for the balanced Clifford volume give
\[
 \rho_{n+4}=\left(1-\frac1{(n+2)^2}\right)\rho_n
 \qquad\text{for even }n.
\]
For odd dimensions, the sphere-volume formula gives $
 \rho_{2k+1}
 =2^{k+1}
 \left(\frac{k}{2k+1}\right)^{k/2}
 \left(\frac{k+1}{2k+1}\right)^{(k+1)/2}.$ And, for the continuous extension of the odd-dimensional formula,
\[
 \frac{d}{dk}\log\rho_{2k+1}
 =\frac12\log\frac{4k(k+1)}{(2k+1)^2}<0.
\]
Thus the asserted bounds reduce to the initial values
\[
 \rho_4=\frac32,\quad \rho_5,\rho_6<\frac32,
 \qquad
 \rho_{12}=\frac{231}{160},\quad
 \rho_{13},\rho_{14}<\frac{231}{160}.
\]
These finite comparisons follow directly from the volume formulas and are also checked  by the exact-arithmetic code available in Appendix~\ref{sec:certificates}.
\end{proof}

\begin{theorem}\label{thm:clifford-volume-gap}
Let $n\geq3$ and let $f:M^n\to\sph^N$ be a closed, linearly full minimal
immersion. Set
$N_*(n)=9$ when $n=3$ and $N_*(n)=3n-1$ when $n\geq4.$ If $N\geq N_*(n)$, then
\begin{equation}\label{eq:clifford-volume-gap}
 \Vol(f)>V_n^{\mathrm{Cl}}.
\end{equation}
\end{theorem}
 \begin{proof}
By \eqref{eq:coordinate-spectrum}, linear fullness gives
$\lambda_{N+1}\leq n$. Set $m_n=N_*(n)+1$.
Since $N\geq N_*(n)$, we have $\lambda_{m_n}\leq n$.

For $3\leq n\leq13$, take $t_n=3/(n+4)$.
The exact-arithmetic code available in Appendix~\ref{sec:certificates}
gives
\[
 \frac{G_{n,m_n}(t_n)}{Z_n(t_n)}>1.56,
 \qquad 3\leq n\leq13.
\]
Since $\rho_3=8/(3\sqrt3)<1.54$ and $\rho_n\leq3/2$ for $n\geq4$,
Theorems~\ref{thm:general-heat} and~\ref{thm:cly} imply
\[
 \frac{\Vol(f)}{\omega_n}
 >\frac{G_{n,m_n}(t_n)}{Z_n(t_n)}>\rho_n.
\]

For $n\geq14$, we have $m_n=3n$.
Taking $m=3n$ and $x=2/(3n)$ in
\eqref{eq:clifford-basic-volume} yields
\[
 \frac{\Vol(f)}{\omega_n}>h(n),\qquad
 h(u)\triangleq 3\frac{(1-2/(3u))^u}{1+2/(3u)}.
\]
The function $h$ is increasing for $u>2/3$: its numerator is increasing
and its positive denominator is decreasing. Indeed, 
\[
 \frac{d}{du}\left[u\log\left(1-\frac{2}{3u}\right)\right]
 =\log\left(1-\frac{2}{3u}\right)+\frac{2/(3u)}{1-2/(3u)}>0,
\]
The same code gives $h(14)>1.446$.
By Lemma~\ref{lem:clifford-volume-ratio},
\[
 h(n)\geq h(14)>1.446>\frac{231}{160}\geq\rho_n.
\]
This completes the proof.
\end{proof}


\appendix
\section{Estimates for the explicit constants}\label{sec:certificates}

Retain the notation $G_m=G_{2,m}$.  
\begin{lemma}\label{lem:cseven}
The fixed-time constants in \eqref{eq:cert-fixed-times} satisfy
\[
  2\pi^2<C_7<20,\qquad
  C_{10}>\frac{236}{29}\pi>8\pi,\qquad
  C_8>\frac{48\pi}{7}>6\pi E\left(\frac{2\sqrt2}{3}\right).
\]
\end{lemma}
\begin{proof}
{With the aid of exact-arithmetic computation in Mathematica, we find} 
\begin{equation}\label{eq:surface-ratio-bounds}
  1.573<\frac{G_7(1/2)}{Z_2(1/2)}<1.574,\quad
  1.732<\frac{G_8(1/2)}{Z_2(1/2)}<1.733,\quad
  2.069<\frac{G_{10}(2/5)}{Z_2(2/5)}<2.070.
\end{equation}
Together with $\pi<22/7$, these inequalities prove the assertions
for $C_7$ and $C_{10}$, as well as $C_8>48\pi/7$.
For the remaining comparison, Taylor's formula gives
$\sqrt{1-y}\leq1-y/2-y^2/8-y^3/16$ for $0\leq y<1$.
Integrating with $y=k^2\sin^2\theta$ yields, for $0\leq k<1$,
\[
 E(k)\leq\frac\pi2
 \left(1-\frac{k^2}{4}-\frac{3k^4}{64}-\frac{5k^6}{256}\right).
\]
At $k^2=8/9$, this upper bound is less than $8/7$, again using
$\pi<22/7$.
\end{proof}

\medskip
\noindent\textit{Exact verification in Mathematica.}
All terminating decimals appearing in the displayed inequalities denote exact
rational numbers.  A source file is available on the website doi.org/10.13140/RG.2.2.20189.99044, where the code verifies the comparisons in
Lemma~\ref{lem:cseven}, the endpoint comparisons in
Lemma~\ref{lem:clifford-volume-ratio}, and the finite comparisons in the
proof of Theorem~\ref{thm:clifford-volume-gap}.
It uses rational Taylor bounds for the exponential function and a
geometric-series bound for the tail of $Z_n(t)$. The error bounds,
input parameters, and individual tests are specified within the code;
no floating-point sign determination is used. The extension to all $n\geq14$
relies on the analytic monotonicity argument in the main text. 

 \vspace{2mm}

{\bf{Declaration on the use of AI:}} 
 The OpenAI tool ChatGPT was used under the authors’ direction to assist with mathematical exploration and paper writing. The authors wrote the paper, independently verified all AI-assisted content, and take full responsibility for the contents of the paper. \vspace{2mm}


{\bf {Funding~Declarations:}} PW is supported by the Natural Science Foundation of Fujian Province of China (Grant No. 2026J002028) and  NSFC (Grant No. 12371052). ZXX is supported by NSFC (Grant No. 12671061) and the Fundamental Research Funds for Central Universities.

\end{document}